\documentclass{amsart}
\usepackage{amssymb,amsfonts,amsmath}
\usepackage{graphicx}
\usepackage{epsfig}
\usepackage{float}
 \usepackage{tikz}
\usepackage{subfigure}
\usepackage{hyperref}
\hypersetup{
    colorlinks,%
    citecolor=blue,%
    filecolor=black,%
    linkcolor=black,%
    urlcolor=blue
   }

\newtheorem{theorem}{Theorem}[section]
\newtheorem{lemma}{Lemma}[section]

\numberwithin{equation}{section} \numberwithin{theorem}{section}

\newcounter{labelflag} 

\newcommand{\Label}[1]{
                       \ifnum\thelabelflag=1
                          \ifmmode
                            \makebox[0in][l]{\qquad\fbox{\rm #1}}
                          \else
                             \marginpar{\vspace{0.7\baselineskip}
                                        \hspace{-1.1\textwidth}
                                        \fbox{\rm#1}}
                          \fi
                       \fi
                       \label{#1}
                      }

\def\bt{\begin{thm}}
\def\et{\end{thm}}

\def\bl{\begin{lem}}
\def\el{\end{lem}}

\def\bd{\begin{defi}}
\def\ed{\end{defi}}

\def\bc{\begin{cor}}
\def\ec{\end{cor}}

\def\bp{\begin{proof}}
\def\ep{\end{proof}}

\def\br{\begin{rem}}
\def\er{\end{rem}}

\newtheorem{thm}{Theorem}[section]
\newtheorem{lem}{Lemma}[section]
\newtheorem{defi}{Definition}[section]
\newtheorem{rem}{Remark}[section]
\newtheorem{cor}{Corollary}[section]

\newcommand{\be}{\begin{equation}}
\newcommand{\ee}{\end{equation}}

\renewcommand{\det}{\mathrm{det}}

\newcommand{\N}{\mathbb N}
\newcommand{\Z}{\mathbb Z}

\newcommand{\G}{\mathcal{H}}
\newcommand{\eL}{\mathcal{L}}
\newcommand{\C}{\mathcal{C}}

\newcommand{\af}{\mathfrak{a}}

\newcommand{\bif}{\mathfrak{b}}

\newcommand{\vs}{\vspace{10pt}}
\def\sg{\mathop{\rm sng}}

\graphicspath{{Chemo_figs/}}
\begin{document}

\title[Transitions in Chemotaxis]{Phase transition and hexagonal patterns in rich stimulant diffusion-chemotaxis model}

\author[M. Yari]{Masoud Yari}
\address[MY]{The Department of Mathematics \& Statistics
Texas A\&M University-Corpus Christi,
Corpus Christi, TX 78412-5825} \email{masoud.yari@tamucc.edu}

\begin{abstract}
An important component in studying mathematical models in many biochemical systems, such as those found in developmental biology, is phase transition. The purpose of this work is to analyze the phase transition property of a diffusion-chemotaxis model with proliferation source, as a macroscopic model of behavior of mobile species. Along the way, we will discuss that the system exhibits very rich pattern-forming behavior.  In particular, a portion of the present work is devoted to the proof of existence of hexagonal patterns as a result of instability of two Fourier modes. It is also shown that they are either saddle points or attracting nodes. Moreover, they belong to an attractor which consists of finite number of steady-state  solutions and their connecting heteroclinic orbits. The structure of this attractor will be precisely determined as well.
\end{abstract}

\maketitle

\today

\section{Introduction}
Macroscopic study of pattern-forming and clustering properties of mobile species starts with examining their collective behavior and applying the physical laws to drive mathematical models. From this perspective, the dynamics of any mobile species in space is characterized  by elements such as concentration, proliferation, degradation, and random and directional movement, following physical rules such as the conservation principle.

In the study of mobile biological species, directional movement based on chemical signals, known as chemotaxis, plays an important role. A pioneer mathematical model of chemotaxis is due to Keller and Segal in 1970 \cite{KS70,KS71}. They proposed a system of four strongly coupled parabolic differential equations to describe aggregation of cellular slim molds.  Since then their model and its variants have been the subject of many studies, see \cite{murray-b2, horts03} and references therein.

The focus of this work will be on a main variant of  the classic Keller-Segal model. Based on previous studies of
the semi-solid medium experiments \cite{TLM99, BBH91}, this work will focus on the diffusion-chemotaxis models with proliferation source
\begin{equation}\label{eq.intro}
    \begin{split}
   \frac{\partial {{u}_{1}}}{\partial t}&={{d}_{1}}\Delta {{u}_{1}}-\chi \nabla ({{u}_{1}}\nabla {{u}_{2}})+ f({{u}_{1}}), \\
  \frac{\partial {{u}_{2}}}{\partial t}&={{d}_{2}}\Delta{{u}_{2}}-a{{u}_{2}}+b {{u}_{1}}; \\
    \end{split}
\end{equation}
where $u_1$ is the population density of biological individuals, and $u_2$ is the chemoattractant concentration; system parameters $d_i$ and $\chi$ are positive constants representing the diffusion and chemotactic coefficients respectively; and $a$ and $b$ are positive constants. Here $f$  is a nonlinear proliferation source term. The equation is defined on a rectangular domain $\Omega=(0,\ell_1)\times(0,\ell_2)$ and is supplemented with no-flux boundary conditions.

An essential component in pattern-forming behavior of the modified KS system is phase transition. The presentence of the chemotaxis source term, $f$, will allow finite amplitude patterns to continue to exist in a longer time-scale \cite{smoller94}. However, persistence of patterns is not rigorously known. In this work we follow the general framework proposed by T. Ma and S. Wang  in their studies of phase transitions of dissipative systems \cite{mw-b2}.
Strongly motivated by phase transition problems in nonlinear sciences, their dynamic transition theory aims at finding a full set of transition states.
The set of transition states is represented by a local attractor near or away from the basic state.

In fact, the fundamental element of their theory is the introduction of a dynamic classification scheme for phase transitions. Dynamic transitions are classified into three types: continuous (Type-I), jump (Type-II) and mix (Type-III).  The greatest advantage of this classification scheme and the related dynamic transition theory is that it provides a complete set of the transition states and their dynamic properties.  Once the type of the dynamic transition is determined for a given equilibrium system, the order of transitions in the classical sense immediately becomes transparent, leading to precise understanding of the underlying physical system.	

Another important feature of the present work concerns the existence and persistence of hexagonal patterns. In fact we will derive the necessary conditions for formation of hexagonal patterns when two modes become unstable. In this work we will derive transition equations which give a larger picture of the dynamics; this study will underline the transient states and include the final steady state bifurcated solutions from the trivial solution; thus we will be able to understand the connection between these components since  the transition equations will provide a comprehensive picture of the main qualitative dynamics.

It should be mentioned that the bifurcation analysis of the  diffusion-chemotaxis models with proliferation source has been conducted in several  works, see \cite{KO2012} (and references therein) where the bifurcation problem of the stationary system has been studied; also see \cite{zhumurray95, maini91} and references therein.  It is also known that hexagonal patterns can emerge from three unstable modes with no restriction on the geometry of the spatial domain; and, in fact, the existence of hexagonal patterns with three unstable modes has been recently studied in \cite{OO2011}.

\section{The rich stimulant Keller-Segel model}

Consider a bacterial species which moves in a medium where two main components lead its macroscopic behavior, namely a chemoattractant and a stimulant.  Assume that  $u_1$ is the population density of biological individuals, $u_2$ is the chemoattractant concentration, and $u_3$ is the stimulant concentration. It is discussed that the physical laws will lead us to the following model which is a modification of the original Keller-Segel system:
\begin{equation}\label{}
    \begin{split}
   \frac{\partial {{u}_{1}}}{\partial t}&={{d}_{1}}\Delta {{u}_{1}}-\chi \nabla ({{u}_{1}}\nabla {{u}_{2}})+f({{u}_{1}},{{u}_{2}},{{u}_{3}}), \\
  \frac{\partial {{u}_{2}}}{\partial t}&={{d}_{2}}\Delta{{u}_{2}}+{{r}_{1}}{{u}_{1}}-{{r}_{2}}{{u}_{2}}, \\
  \frac{\partial {{u}_{3}}}{\partial t}&={{d}_{3}}\Delta {{u}_{3}}-{{r}_{3}}{{u}_{1}}{{u}_{3}}+q( x),
    \end{split}
\end{equation}

where $f({{u}_{1}},{{u}_{2}},{{u}_{3}})={{\alpha }_{1}}{{u}_{1}}\left( \frac{{{\alpha }_{2}}{{u}_{3}}}{{{\alpha }_{0}}+{{u}_{3}}}-u_{1}^{2} \right)$, and  $\chi$, $\alpha_i$'s, $d_i$'s,  and $r_i$'s are all positive constants, and $q(x)$  is the nutrient source.

When the stimulant $u_3$ is ample, that is $u_3=\infty$, the last equation in the above system can be ignored. We also have
	\[f={{\alpha }_{1}}{{\alpha }_{2}}{{u}_{1}}-{{\alpha }_{1}}u_{1}^{3}={{\alpha }_{1}}{{u}_{1}}({{\alpha }_{2}}-u_{1}^{2}).\]
The following  change of variables
\begin{equation}\label{}
    \begin{split}
  & t=\frac{t'}{{{r}_{2}}},\quad x=\sqrt{\frac{{{d}_{2}}}{{{r}_{2}}}}x',\quad {{u}_{1}}=\sqrt{{{\alpha }_{2}}}u_{1}^{'},\quad {{u}_{2}}=\frac{{{d}_{2}}}{\chi }u_{2}^{'}, \\
 & \lambda =\frac{{{r}_{1}}\sqrt{{{\alpha }_{2}}}\chi }{{{r}_{2}}{{d}_{2}}},\quad \alpha =\frac{{{\alpha }_{1}}{{\alpha }_{2}}}{{{r}_{2}}},\quad \mu =\frac{{{d}_{1}}}{{{d}_{2}}},
    \end{split}
\end{equation}
result in a non-dimensionalized system (after dropping the primes):
\begin{equation}
    \begin{split}
  & \frac{\partial {{u}_{1}}}{\partial t}=\mu \Delta {{u}_{1}}-\nabla ({{u}_{1}}\nabla {{u}_{2}})+\alpha f({{u}_{1}}), \\
 & \frac{\partial {{u}_{2}}}{\partial t}=\Delta {{u}_{2}}-{{u}_{2}}+\gamma {{u}_{1}},
    \end{split}
    \Label{eq.UV}
\end{equation}
where $f(u_1)= u_1 (1 -  u_1^2)$.  We will consider the above equation on a rectangular domain $\Omega=(0,\ell_1) \times (0,\ell_2)$  with the Neumann boundary condition and typical initial conditions:
	\[{{\left. \frac{\partial ({{u}_{1}},{{u}_{2}})}{\partial n} \right|}_{\partial \Omega }}=0,\qquad ({{u}_{1}},{{u}_{2}})(x,0)=(u_{1}^{0},u_{2}^{0})(x).\]
Here we note that $(\bar{u_1},\bar{u_2})=(1,\gamma)$ is a nontrivial uniform solution of (\ref{eq.UV}). Therefore by a change of variable we `center' the equation around $(\bar{u_1},\bar{u_2})$ to obtain
\begin{equation}
\begin{array}{l}
{\frac{{\partial u}}{{\partial t}} = \mu \Delta u - \nabla (u\nabla v) - \gamma \Delta v + \alpha u(1 - {u^2}),}\\
\frac{{\partial v}}{{\partial t}} = \Delta v - v + \gamma u,\\
{\left. {\frac{{\partial (u,v)}}{{\partial n}}} \right|_{\partial \Omega }} = 0\, \, ,\qquad (u,v)(x,0) = ({u_0},{v_0})(x);
\end{array}\Label{eq.uv}
\end{equation}
where
\[u=u_1-\bar{u_1},\qquad v=u_2-\bar{u_2}.\]

Throughout this work, we will focus on an important case where the diffusion and degradation of the chemoattractant  by the bacteria themselves are
almost balanced by their production. This renders the second equation in (\ref{eq.uv}) to the following stationary equation
\begin{equation}
0 = \Delta v - v + \gamma u,
\end{equation}
which gives us
\[v=[-\Delta+1]^{-1}u.\]
Therefore, we drive the following equation
\begin{equation}
 {u_t} = \eL(u) + \G(u),
 \Label{eq.u}
\end{equation}
where
\begin{equation}
\begin{array}{l}\vs
\eL(u) = \mu \Delta u - 2\alpha u - \gamma \Delta {[ - \Delta  + I]^{ - 1}}u,\\
\G = {\G_2} + {\G_3},
\end{array}\Label{def.LG}
\end{equation}
with
\[
\begin{array}{l} \vs
{\G_2}(u) =  - \gamma \nabla u \nabla ({[ - \Delta  + I]^{ - 1}}u) - \gamma u\Delta ({[ - \Delta  + I]^{ - 1}}u) - 3\alpha {u^2},\\
{\G_3}(u) =  - \alpha {u^3}.
\end{array}
\]
It is easy to see that the eigenvalue problem
\[\eL e = \sigma e,\]
with no-flux boundary condition yields these eigenvalues and eigenvectors:
\begin{equation}
\begin{split}
e_k &=\cos\left(\frac{k_1\pi x_1}{\ell_1}\right) \cos\left(\frac{k_2\pi x_2}{\ell_2}\right),\\
\sigma_k &=\sigma(\rho_k) = - \mu {\rho _k} - 2\alpha  + \gamma \frac{{{\rho _k}}}{{1 + {\rho _k}}};
\Label{def.eval}
\end{split}
\end{equation}
where $k=(k_1,k_2) \in \Z^{+} \times \Z^{+}$ and $\rho_k=\pi^2\left(\frac{k_1^2}{\ell_1^2}+\frac{k_2^2}{\ell_2^2}\right).$

\subsection{Existence and Transitions of Hexagonal Patterns}
\begin{figure}
\subfigure[Regular Hexagon]{
  \includegraphics[width=.3\textwidth]{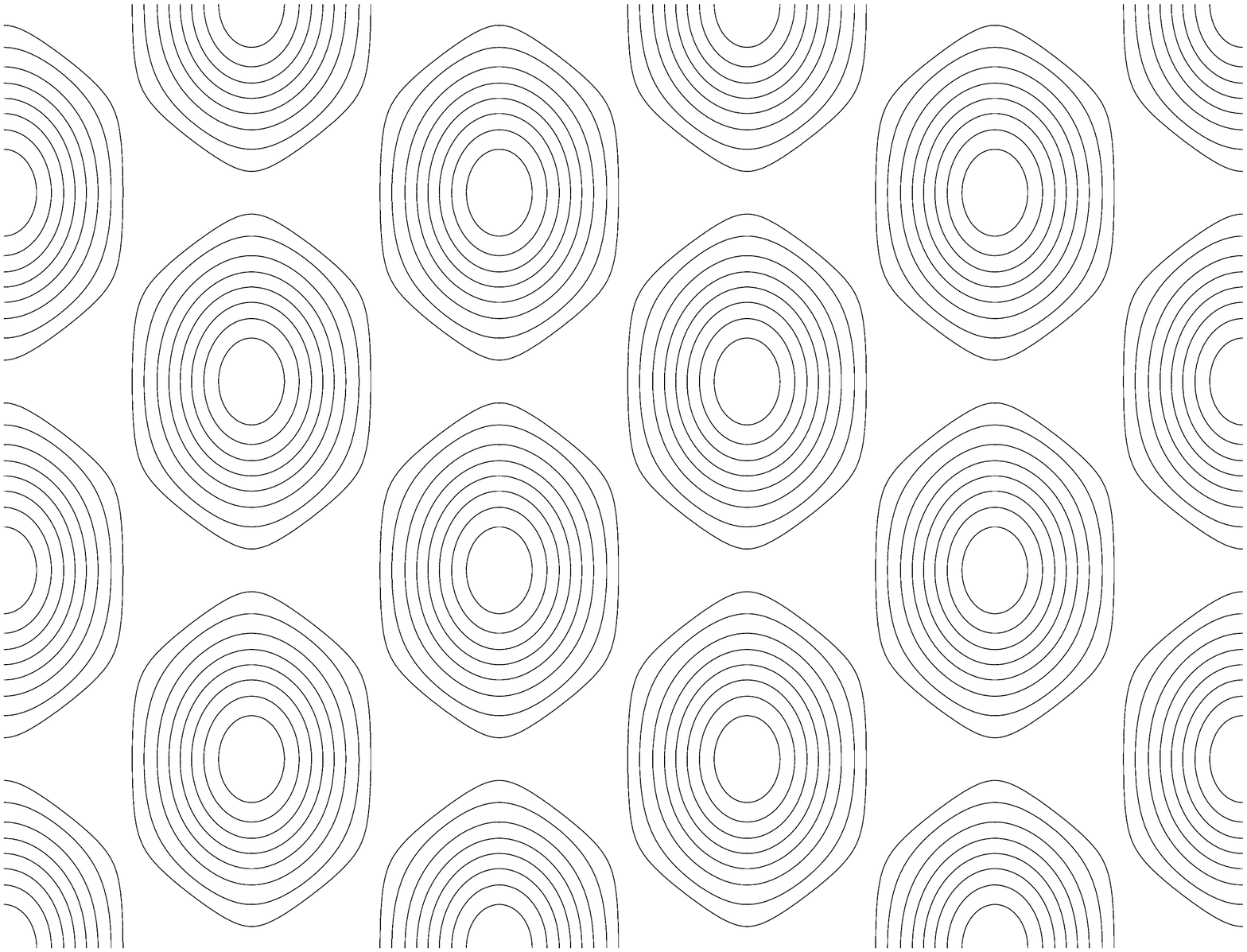}
  \label{fig1.hex}}
  \subfigure[Rectangular]{
  \includegraphics[width=.3\textwidth]{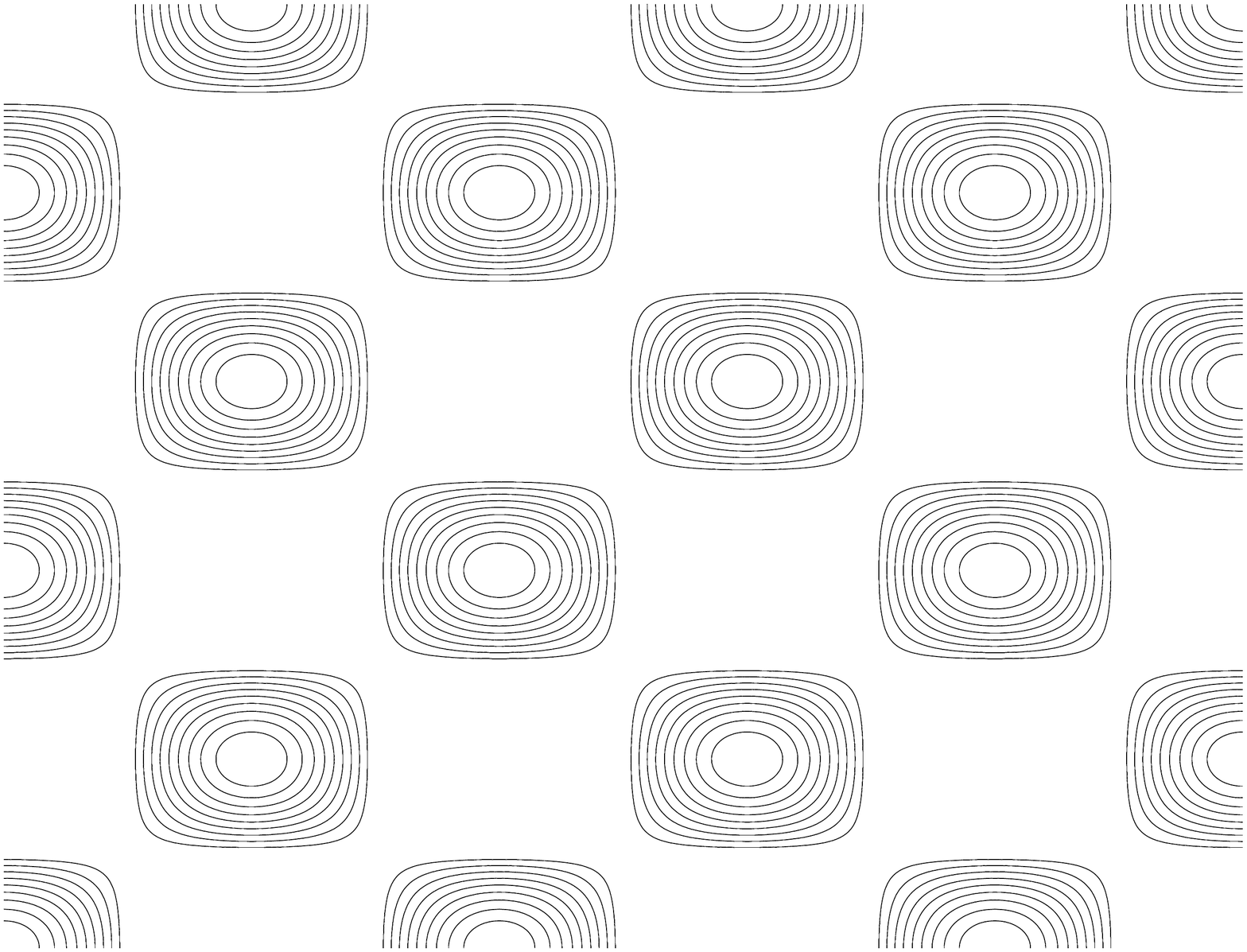}
  \label{fig1.rec}}
  \subfigure[Strip]{
  \includegraphics[width=0.3\textwidth]{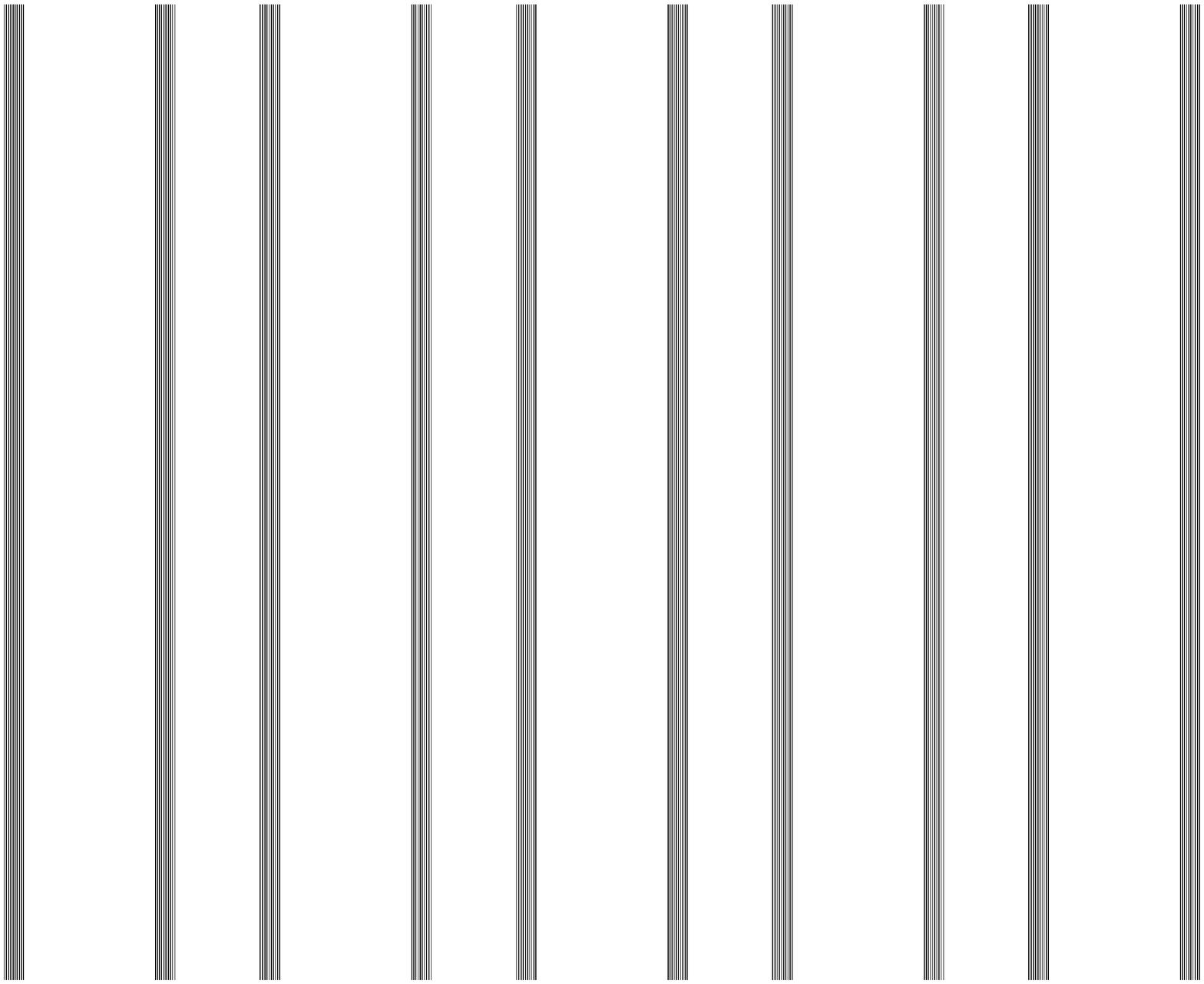}
  \label{fig1.strip}}
  \caption{Spontaneously merged patterns.}
\label{fig.1}
\end{figure}
One of the main concerns of this work is the study of the formation and persistence of hexagonal patterns in a rectangular spatial domain. It is known that in rectangular domains, hexagonal patterns can be obtained from two critical modes only if the aspect ratio of the spatial domain is irrational; precisely, for this to happen we need to assume that
\begin{equation}
 \sqrt{3} n \ell_1= m \ell_2,
    \Label{cond.ratio}
\end{equation}
for some $m, n \in \N$ such that $m$ and $n$ are relatively prime.

Now we can state our first main theorem:
\begin{theorem}\label{th.1}
Suppose
\begin{equation}
   \ell_2=2\sqrt{2}n \pi, \quad \mu=8 \alpha
\Label{cond.th1}
\end{equation}
Then solutions of problem \ref{eq.u} bifurcate at $4\lambda=9\mu$. The transition of Problem~\eqref{eq.u}  at  $4\lambda=9\mu$  is continues, that is, Type-I. Moreover, there exist $\varepsilon>0$ such that
\begin{itemize}
\item When $\frac94-\varepsilon <\frac\lambda\mu\le\frac94$, the trivial solution of problem~\ref{eq.u} is an asymptotically stable equilibrium point (Fig.\ref{fig.I.l}).
\item When $\frac94<\frac\lambda\mu<\frac94+\varepsilon$, problem~\ref{eq.u} has eight equilibrium points which are regular. There exist two strip, two rectangular, and four regular hexagonal patterns (Figure~\ref{fig.1}). The patterns are given by $u=y_1 e_{m,n}+y_2 e_{0,2n}+o(2)$ where $y=(y_1,y_2)$ is one of the following values
\begin{equation}
    \begin{split}
\pm Y_s& \sim (0,\pm 1)\sqrt {{\frac{\sigma}{\mu}}}, \\ 
 \pm Y_r& \sim (\pm 1,0)\sqrt{{\frac{\sigma}{\mu}}}, \\ 
\pm Y_h^{\pm} & \sim \pm (\pm 2,1) \sqrt {{\frac{\sigma}{\mu}}}.
    \end{split}
    \Label{def.ep}
\end{equation}
with $\sigma=\sigma_{m,n}$.
\item All eight equilibrium  points  together with their connecting orbits form an attractor which is homeomorphic to $S^1$ (Fig.\ref{fig.I.r}).
\item Hexagonal patterns are stable nodes; strip and rectangular points are saddle points (Fig.\ref{fig.I.r}).
\end{itemize}
\end{theorem}
The proof of this theorem will be provided later in this paper.

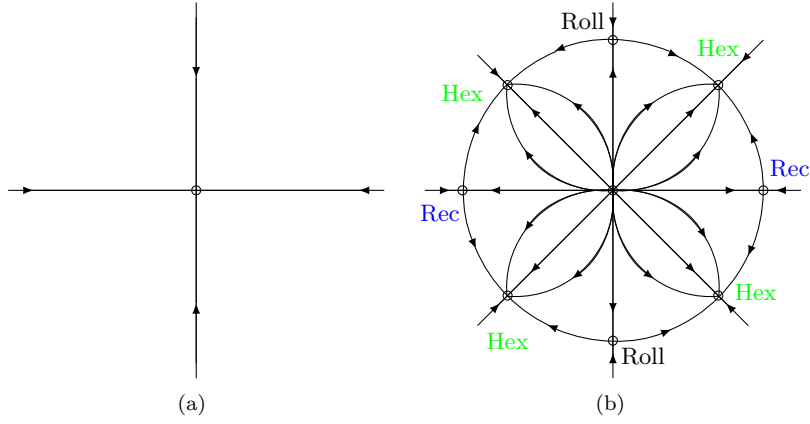
\begin{figure}
 \subfigure[]{
  \begin{tikzpicture}
   \begin{scope}[>=latex]
   \draw (2.5,0) -- (2.5,5);\draw (2.5,4) -- (2.5,1);
   \draw [<-](2.5,4) -- (2.5,4.8);
   \draw [->](2.5,0.2) -- (2.5,1);
    \draw (0,2.5) -- (5,2.5);\draw (1,2.5) -- (4,2.5);
   \draw [>-](0.2,2.5) -- (1,2.5);
   \draw [-<](4,2.5) -- (4.8,2.5);
\draw (2.5,2.5) circle (.06cm);
 	\end{scope}
\end{tikzpicture}
 \label{fig.I.l}}
  \subfigure[]{
  \begin{tikzpicture}
   \begin{scope}[>=latex]
   \draw (2.5,0) -- (2.5,5);\draw (2.5,4) -- (2.5,1);
   \draw [>-<](2.5,4) -- (2.5,4.8);
   \draw [>-<](2.5,0.2) -- (2.5,1);
    \draw (0.7,4.3) -- (4.3,0.7);\draw (1.5,3.5) -- (3.5,1.5);
   \draw [>-<](0.9,4.1) -- (1.5,3.5);
   \draw [>-<](3.5,1.5) -- (4.1,0.9);
    \draw (0,2.5) -- (5,2.5);\draw (1,2.5) -- (4,2.5);
   \draw [>-<](0.2,2.5) -- (1,2.5);
   \draw [>-<](4,2.5) -- (4.8,2.5);
    \draw (0.7,0.7) -- (4.5,4.5);\draw (1.5,1.5) -- (3.5,3.5);
   \draw [>-<](0.9,0.9) -- (1.5,1.5);
   \draw [>-<](3.5,3.5) -- (4.3,4.3);
   \draw [->] (4.34,1.71) arc (-22.5:23.5:2cm) ;
   \draw [-<] (4.33,3.25) arc (22.5:67.5:2cm) ;
   \draw [->] (3.25,4.35) arc (67.5:112.5:2cm) ;
   \draw [-<] (1.75,4.35) arc (112.5:157.5:2cm) ;
   \draw [->] (0.67,3.28) arc (157.5:203.5:2cm) ;
   \draw [-<] (0.67,1.74) arc (202.5:247.5:2cm) ;
   \draw [->] (1.75,0.65) arc (247.5:292.5:2cm) ;
   \draw [-<] (3.252,0.64) arc (292.5:338.5:2cm)
   node at (1.1,0.5)[green]{\small Hex} 
   node at (0.2,2.2)[blue]{\small Rec}
   node at (0.5,3.8)[green]{\small Hex} 
   node at (2.1,4.75){\small Roll}
   node at (3.9,4.4)[green]{\small Hex} 
   node at (4.85,2.8)[blue]{\small Rec}
   node at (4.4,1.15)[green]{\small Hex}
   node at (2.9,0.3){\small Roll};
\draw (2.5,2.5) circle (.06cm);
\draw (0.5,2.5) circle (.06cm);
\draw (2.5,0.5) circle (.06cm);
\draw (4.5,2.5) circle (.06cm);
\draw (2.5,4.5) circle (.06cm);
\draw (3.9,3.9) circle (.06cm);
\draw (3.9,1.1) circle (.06cm);
\draw (1.1,1.1) circle (.06cm);
\draw (1.1,3.9) circle (.06cm);
   \draw [->] (2.5,2.5) arc (185:125:1.3cm);
   \draw [-] (2.5,2.5) arc (185:85:1.3cm);
   \draw [->] (2.5,2.5) arc (-95:-35:1.3cm) ;
   \draw [-] (2.5,2.5) arc (-95:5:1.3cm);

      \draw [->] (2.5,2.5) arc (275:215:1.3cm);
     \draw [-] (2.5,2.5) arc (275:175:1.3cm);
     \draw [->] (2.5,2.5) arc (-5:55:1.3cm) ;
    \draw [-] (2.5,2.5) arc (-5:95:1.3cm);

      \draw [->] (2.5,2.5) arc (5:-55:1.3cm);
     \draw [-] (2.5,2.5) arc (5:-95:1.3cm);
     \draw [->] (2.5,2.5) arc (85:145:1.3cm) ;
     \draw [-] (2.5,2.5) arc (85:185:1.3cm);
   \draw [->] (2.5,2.5) arc (95:30:1.3cm);
   \draw [-] (2.5,2.5) arc (95:-5:1.3cm);
   \draw [->] (2.5,2.5) arc (-185:-125:1.3cm) ;
   \draw [-] (2.5,2.5) arc (-185:-85:1.3cm);
%

 	\end{scope}
\end{tikzpicture}
 \label{fig.I.r}}
   \caption[]{Type I transition (structure of attractors):
\subref{fig.I.l} Phase portrait when $\lambda<\lambda_c$;
\subref{fig.I.r} The $S^1$-attractor including  saddle  hexagonal patterns when $\lambda>\lambda_c$.
}
\label{fig.I}
\end{figure}

\begin{figure}
 \subfigure[]{
   \begin{tikzpicture}
   \begin{scope}[>=latex]
   \draw (2.5,0) -- (2.5,5);\draw (2.5,4) -- (2.5,1);
   \draw [>-<](2.5,4) -- (2.5,4.8);
   \draw [>-<](2.5,0.2) -- (2.5,1);
    \draw (0.7,4.3) -- (4.3,0.7);\draw (1.5,3.5) -- (3.5,1.5);
   \draw [>-<](0.9,4.1) -- (1.5,3.5);
   \draw [>-<](3.5,1.5) -- (4.1,0.9);
       \draw (0,3) -- (2.5,2.5); \draw (2.5,2.5) -- (5,3);
   \draw [>-<](0.2,2.96) -- (1,2.8);
   \draw [>-<](4,2.8) -- (4.8,2.96);
    \draw (0.7,0.7) -- (4.5,4.5);\draw (1.5,1.5) -- (3.5,3.5);
   \draw [>-<](0.9,0.9) -- (1.5,1.5);
   \draw [>-<](3.5,3.5) -- (4.3,4.3);
   \draw [->] (4.34,1.71) arc (-22.5:23.5:2cm) ;
   \draw [-<] (4.33,3.25) arc (22.5:67.5:2cm) ;
   \draw [->] (3.25,4.35) arc (67.5:112.5:2cm) ;
   \draw [-<] (1.75,4.35) arc (112.5:157.5:2cm) ;
   \draw [->] (0.67,3.28) arc (157.5:203.5:2cm) ;
   \draw [-<] (0.67,1.74) arc (202.5:247.5:2cm) ;
   \draw [->] (1.75,0.65) arc (247.5:292.5:2cm) ;
   \draw [-<] (3.252,0.64) arc (292.5:338.5:2cm)
   node at (1.1,0.5)[green]{\small Hex} 
   node at (0.2,2.5)[blue]{\small Mix}
   node at (0.5,3.8)[green]{\small Hex} 
   node at (2.1,4.75){\small Roll}
   node at (3.9,4.4)[green]{\small Hex} 
   node at (4.85,2.5)[blue]{\small Mix}
   node at (4.4,1.15)[green]{\small Hex}
   node at (2.9,0.3){\small Roll};
\draw (2.5,2.5) circle (.06cm);
  \draw (0.55,2.89) circle (.06cm);
\draw (4.45,2.89) circle (.06cm);
\draw (2.5,0.5) circle (.06cm);
\draw (2.5,4.5) circle (.06cm);
\draw (3.9,3.9) circle (.06cm);
\draw (3.9,1.1) circle (.06cm);
\draw (1.1,1.1) circle (.06cm);
\draw (1.1,3.9) circle (.06cm);

 	\end{scope}
\end{tikzpicture}
 \label{fig.II.r1}
}
\subfigure[]{
  \begin{tikzpicture}
   \begin{scope}[>=latex]
   \draw (2.5,0) -- (2.5,5);\draw (2.5,4) -- (2.5,1);
   \draw [>-<](2.5,4) -- (2.5,4.8);
   \draw [>-<](2.5,0.2) -- (2.5,1);
    \draw (0.7,4.3) -- (4.3,0.7);\draw (1.5,3.5) -- (3.5,1.5);
   \draw [>-<](0.9,4.1) -- (1.5,3.5);
   \draw [>-<](3.5,1.5) -- (4.1,0.9);
    \draw (0,3) -- (2.5,2.5); \draw (2.5,2.5) -- (5,3);
   \draw [>-<](0.2,2.96) -- (1,2.8);
   \draw [>-<](4,2.8) -- (4.8,2.96);
    \draw (0.7,0.7) -- (4.5,4.5);\draw (1.5,1.5) -- (3.5,3.5);
   \draw [>-<](0.9,0.9) -- (1.5,1.5);
   \draw [>-<](3.5,3.5) -- (4.3,4.3);
   \draw [-<] (4.34,1.71) arc (-22.5:23.5:2cm) ;
   \draw [->] (4.33,3.25) arc (22.5:67.5:2cm) ;
   \draw [-<] (3.25,4.35) arc (67.5:112.5:2cm) ;
   \draw [->] (1.75,4.35) arc (112.5:157.5:2cm) ;
   \draw [-<] (0.67,3.28) arc (157.5:203.5:2cm) ;
   \draw [->] (0.67,1.74) arc (202.5:247.5:2cm) ;
   \draw [-<] (1.75,0.65) arc (247.5:292.5:2cm) ;
   \draw [->] (3.252,0.64) arc (292.5:338.5:2cm)
   node at (1.1,0.5)[red]{\small Hex} 
   node at (0.2,2.5)[blue]{\small Mix}
   node at (0.5,3.8)[red]{\small Hex} 
   node at (2.1,4.75){\small Roll}
   node at (3.9,4.4)[red]{\small Hex} 
   node at (4.85,2.5)[blue]{\small Mix}
   node at (4.4,1.15)[red]{\small Hex}
   node at (2.9,0.3){\small Roll};

\draw (2.5,2.5) circle (.06cm);
\draw (0.55,2.89) circle (.06cm);
\draw (4.45,2.89) circle (.06cm);
\draw (2.5,0.5) circle (.06cm);
\draw (2.5,4.5) circle (.06cm);
\draw (3.9,3.9) circle (.06cm);
\draw (3.9,1.1) circle (.06cm);
\draw (1.1,1.1) circle (.06cm);
\draw (1.1,3.9) circle (.06cm);
\end{scope}
\end{tikzpicture}
 \label{fig.II.r2}
 }
   \caption[]{Type I transition: the structure of the attractor $\mathcal{A}$ and the phase portrait
\subref{fig.II.r1} when $2\bif_2-\bif_1<0$; and
\subref{fig.II.r2} when $2\bif_2-\bif_1>0$.
}
\label{fig.II}
\end{figure}
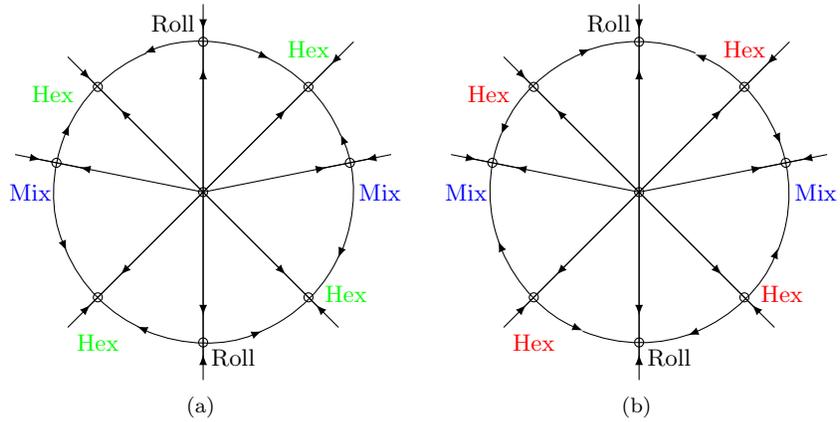
Hexagonal patterns can exist in other parametric domains as well, however the transition type often is a mixture of jump and continuous transitions, i.e., it is Type-III transition.
In order to achieve hexagonal patterns, we choose the spatial geometry in a way that
\begin{equation}
\rho =\sqrt {{\frac {2\alpha}{\mu}}},
\Label{cond.rho}
\end{equation}
where $\rho=\rho_{m,n}$.
We also assume that the physical parameters of the system satisfy the following relationship
\begin{equation}
\lambda=(\sqrt \mu+ \sqrt{ 2\alpha})^2.
\Label{cond.lambda}
\end{equation}

Next we define a critical parameter
\begin{equation}\label{}
\lambda_c=  \min\limits_{\rho _k} \frac{{({\rho _k} + 1)(\mu {\mkern 1mu} {\rho _k} + 2{\mkern 1mu} \alpha )}}{{{\rho _k}}},
\Label{def.lambdac}
\end{equation}
and we let
\begin{equation}
    \C = \left\{ {k = ({k_1},{k_2}) \in {\Z^+} {\text{  such that }}{\rho _k}{\text{  minimizes }}} \eqref{def.lambdac} \right\}.
\Label{def.C}
\end{equation}
In the next theorem we will see that a small perturbation of  $\ell_2$ around $2\sqrt{2}n \pi$ will not change the structure of the attractor, In other words, the attractor is going to be structurally stable at $\ell_2$, but the equilibrium points and patterns can change.
\begin{theorem}\label{th.2}
Let $\ell_c= 2\sqrt{2}n \pi$ and $\lambda_c$ as defined in \eqref{def.lambdac}; then system \ref{eq.u} undergoes a Type-I transition. Moreover, there exist $\varepsilon>0$ such that
\begin{itemize}
\item when $\ell_2=\ell_c$ and $\lambda_c-\varepsilon<\lambda<\lambda_c$, the trivial solution is an asymptotically stable equilibrium point.
\item when $\ell_c<\ell_2<\ell_c+\varepsilon$ and $\lambda_c<\lambda<\lambda_c+\varepsilon$, the trivial solution loses its stability and the solutions bifurcate to exactly eight regular equilibrium points.
     \item The equilibrium points and their transient states form an attractor, homeomorphic to $S^1$; and only one of the following cases can happen
      \begin{itemize}
      \item  There are four stable node hexagonal patterns, two saddle roll patterns, and two saddle mixed patterns (Fig.\ref{fig.II.r1}).
        \item  There are four saddle hexagonal patterns and two stable node roll patterns, and two stable node mixed patterns (Fig.\ref{fig.II.r2}).

      \end{itemize}
\end{itemize}
\end{theorem}
{\rem{\rm Here we introduce a new critical value for the spatial length, $\ell_c$. The first theorem shows that the critical behavior of the system will result in formation of hexagonal, roll, and rectangular patterns as soon as $\lambda$ is perturbed. But as soon as both values of $\lambda$  and  the spatial length, $\ell$, are disturbed then we still  have hexagonal and roll patterns. But there will not ba any rectangular patterns; instead, we will have a mixture of hexagonal and rectangular patterns.}
}
\section{Phase Transition Equations}
We start with provide an appropriate functional setting for our problem \eqref{eq.u}. We consider $\eL$ and $\G$, defined in \eqref{def.LG}, to be
operators $\eL:H\rightarrow H_1$ and $\G : H\rightarrow H_1$ with two Hilbert spaces $H$ and $H_1$ where
\[
H=L^2(\Omega),\qquad H_1=\{u\in H^2(\Omega) | \quad \frac{{\partial {u}}}{{\partial n}} = 0\quad {\rm{on}}\quad \partial \Omega \}.
\]
 In this functional settings, the eigenvalues and eigenvectors of the operator
$\eL$ are as in \eqref{def.eval}.

Obviously $\{e_k\}_k$ provides an orthogonal basis for the Hilbert space $H_1$; therefore, we can write the general solution of the full system \eqref{eq.u} as an infinite series  $u=\sum y_I e_I$.
\subsection{Principle of Exchange of Stability (PES)}
Interesting pattern formation behavior can happen as system \eqref{eq.u} goes through a bifurcation process. As like all bifurcation problems, we need to look for parametric domains where instabilities can occur. In fact conditions \eqref{cond.ratio}, \eqref{cond.rho}, and \eqref{cond.lambda} provide the necessary relation between the parameters. The following lemma formulates what we will witness if these conditions hold.

\begin{lemma}\label{lemma.pes}
Suppose \eqref{cond.ratio}, \eqref{cond.rho}, \eqref{cond.lambda} holds true, and
the set $\C$ and $\lambda_c$ are as defined in \eqref{def.C}  and \eqref{def.lambdac};
then we have
\begin{equation*}
\sigma ({\rho _{{k}}})\left\{ {\begin{array}{ccc}
{ < 0}& \text{ if }&{\lambda  < {\lambda _c,}}\\
{ = 0}& \text{ if }&{\lambda  = {\lambda_c,}}\\
{ > 0}& \text{ if }&{\lambda  > {\lambda_c,}}
\end{array}} \right.
\end{equation*}
when $k \in \C$; and
$
\sigma ({\rho _k}) < 0 {\text{ if }}k \notin \C
$
\end{lemma}

The proof of the above lemma is straightforward.

\subsection{Extracting Transition Equations}
At the critical parameter $\mu=\mu_c$, the solution of \eqref{eq.u} can be described as \[u=u_c+u_s,\] where $u_c$ is the critical leading pattern, and $u_s$ is the ensemble of the slow growing modes. By the center manifold theorem \cite{he84}, we know that $u_s$ is subordinate of  $u_c$.
Being the critical leading pattern, $u_s$ is in fact the combination of fast growing modes $e_I$ where $I \in \C$, that is \[u_c=\sum_{I\in \C} y_I e_I.\]
Based on our assumption, we have
\[
\C=\{\,I_1=(m,n)\,,\,I_2=(0,2n)\,\}.
\]
In fact what we wish to do is to look closely into the dynamics of $u_c$ and how it will be effected by the other slow growing modes. The next lemma provide crucial information about the dynamics of $u_c$. For the sake of simplicity, we let $y_1=y_{I_1}$ and $y_2=y_{I_2}$, where $I_1$ and $I_2$ are members of $\C$ as above. We also let $\sigma_1=\sigma_{I_1}$ and $\sigma_2=\sigma_{I_2}$. We will prove the following lemma first.

\begin{lemma}\label{lemma.rd}
When $\lambda$ lies near $\lambda_c$, the qualitative behavior of system \eqref{eq.u} can be approximated by the following ODE system
\begin{equation}
    \begin{split}
\dot{y}_{1}&=\sigma_1 y_1+ 4\mathfrak{a}y_{1}^{2}+\tfrac{1}{4}({{\mathfrak{b}}_{1}}+2{{\mathfrak{b}}_{2}})y_{1}^{3}
+2{{\mathfrak{b}}_{2}}{{y}_{1}}y_{2}^{2}+o(|y{{|}^{3}}), \\
\dot{y}_{2}&=\sigma_2 y_2+\mathfrak{a}{{y}_{1}}{{y}_{2}}+{{\mathfrak{b}}_{1}}y_{2}^{3}+{{\mathfrak{b}}_{2}}{{y}_{2}}y_{1}^{2}+o(|y{{|}^{3}}).
    \end{split}
    \Label{eq.rd}
\end{equation}
for certain transition numbers $\af$,$\bif_1$, and $\bif_2$, where $y=(y_1,y_2)$.
\end{lemma}

\begin{proof}
We will divide the proof to several steps:
\proof[Step I] We first project the \eqref{eq.u} to the eigenspace of a critical mode $e_K$, $K \in \C$, as follows
\[
\langle {{u_t},{e_K}} \rangle  = \langle {\eL(u) + \G(u),{e_K}} \rangle
= \langle {\eL(u),{e_K}} \rangle  + \langle {\G(u),{e_K}} \rangle.
\]
Therefore, we get
\[{\dot y_K} = {y_K}{\sigma _K} + \frac{{\langle {\G(u),{e_K}} \rangle }}{{\langle {{e_K},{e_K}} \rangle }}.\]
However, the extract a feasible expression from the nonlinear interactions is not easy, specifically, because the strong resonance terms will be degenerate. We have to proceed to a higher order approximation. A straightforward calculation shows that there exist only a finite set of indices, $\C'$, such that on the center manifold we have
\[
\langle {\G(u),{e_K}} \rangle  = \langle {\G(\sum\limits_{I}  {{u_I}{e_I}} ),{e_K}} \rangle
= \langle {\G(u_c+u_{c'}),{e_K}}\rangle +o(|y|^3),
\]
with
\[u_c=\sum\limits_{I\in \C}y_I e_I \text{ and } u_{c'}=\sum\limits_{I\in \C'}y_I e_I.\]
where $y=(y_I)_{I\in \C}$.

In fact, we can see that
\[
\C'=\{\,(0,0)\,,\,(2m,2n)\,,\,(2m,0)\,,\,(0,4n)\,,\,(m,3n)\,\}.
\]
Following a long, but straightforward, calculation, we will drive the following relations
\[\begin{array}{l}
\frac{\langle {\G_2 }(u),{e_{I_1}} \rangle}{\langle e_{I_1},e_{I_1} \rangle}  = 4\af y_{{I_1}}^2
+ \sum\limits_{\substack{I \in \C, I' \in \C'}}
y_I y_{I'}B(I,I',I_1)+o(|y|^3) ,\\
\frac{\langle {{\G_2}(u),{e_{{I_2}}}} \rangle}{\langle e_{I_2},e_{I_2} \rangle}  =  \af {y_{{I_1}}}{y_{{I_2}}}
+ \sum\limits_{\substack{I \in \C, I' \in \C'}}  {{y_I}{y_{I'}}B(I,I',{I_2})} +o(|y|^3);
\end{array}\]
where $y=(y_1,y_2)$,
\begin{equation}
    \af = \tfrac{1}{8}( - 6 \alpha  + \lambda \frac{\rho}{1+\rho}) \text{ with }\rho =\rho_{(m,n)}=\rho_{(0,2n)},\\
     \Label{def.A}
\end{equation}
and
\begin{equation}
B(I,I',K) = P({\rho _I},{\rho _{I'}}) \frac{\langle {{e_I}{e_{I'}},{e_K}} \rangle}{\langle e_{K},e_{K} \rangle} - Q({\rho _I},{\rho _{I'}})
\frac{\langle {\nabla {e_I}\nabla {e_{I'}},{e_K}}\rangle}{\langle e_{K},e_{K} \rangle} ,
\Label{def.B}
\end{equation}
with
\begin{equation*}\label{}
    \begin{split}
P({\rho _I},{\rho _{I'}}) &=  - 6\alpha  + \lambda \left(\frac{\rho _I}{1 + \rho _{I}} +
\frac{\rho _{I'}}{1 + \rho _{I'}}\right),\\
Q({\rho _I},{\rho _{I'}}) &= \lambda \left(\frac{1}{1 + \rho _{I}} + \frac{1}{1 + \rho _{I'}}\right).
\end{split}
\end{equation*}
We then have
\begin{equation}
    \begin{split}
  &\sum\limits_{\substack{{I \in \C}\\{I' \in \C'}}} {{y_I}{y_{I'}}B(I,I',{I_j})}= \\
 &= {y_{{I_1}}}\sum\limits_{I' \in \C'} {{y_{I'}}B({I_1},I',{I_j})}  + {y_{{I_2}}}\sum\limits_{I' \in \C'} {{y_{I'}}B({I_2},I',{I_j})} \\
 &:= {y_{{I_1}}}{B_{1j}} + {y_{{I_2}}}{B_{2j}},
    \end{split}\Label{}
\end{equation}
for $j=1,2$. We note that the following relations hold true;
\begin{equation}
\begin{array}{l}\vs
\left\langle {{e_{m,n}}{e_{2m,0}},{e_{m,n}}} \right\rangle  = \left\langle {{e_{m,n}}{e_{m,3n}},{e_{0,2n}}} \right\rangle  = \left\langle
{{e_{0,2n}}{e_{m,3n}},{e_{m,n}}} \right\rangle ,\\\vs \left\langle {\nabla {e_{m,n}}\nabla {e_{2m,0}},{e_{m,n}}} \right\rangle  = \left\langle
{\nabla {e_{m,n}}\nabla {e_{m,3n}},{e_{0,2n}}} \right\rangle  = \left\langle {\nabla {e_{0,2n}}\nabla {e_{m,3n}},{e_{m,n}}} \right\rangle ,\\\vs
\left\langle {{e_{0,2n}}{e_{0,4n}},{e_{0,2n}}} \right\rangle  = 4\left\langle {{e_{m,n}}{e_{2m,2n}},{e_{m,n}}} \right\rangle ,\\\vs
\left\langle {\nabla {e_{0,2n}}\nabla {e_{0,4n}},{e_{0,2n}}} \right\rangle  = 4\left\langle {\nabla {e_{m,n}}\nabla {e_{2m,2n}},{e_{m,n}}}
\right\rangle ,\\
\left\langle {{e_{0,2n}}{e_{0,4n}},{e_{0,2n}}} \right\rangle  = 2\left\langle {{e_{m,n}}{e_{0,0}},{e_{m,n}}} \right\rangle .
\end{array} \Label{ids}
\end{equation}
Therefore, we can easily see that
\begin{equation}
\begin{aligned}\label{}
    B_{11} &= a\,  {y_{00}} + 2b_2\,  {y_{2m,0}}+cy_{2m,2n}, &&B_{21}= 2b_2\, {y_{m,3n}} ,\\ \vs
    B_{12}&= b_2\, {y_{m,3n}}, &&B_{22} = a\, {y_{00}} + 2b_1\,  {y_{0,4n}};
\end{aligned}\Label{}
\end{equation}
where
\begin{equation}
\begin{array}{l}
a =  P\left( {{\rho _{m,n}},{\rho _{0,0}}} \right),\\
b_1 =
\frac14 P\left( {{\rho _{m,n}},{\rho _{2\, m,2\, n}}} \right) -  \frac{\rho}{2} Q\left( {{\rho _{m,n}},{\rho _{2\, m,2\, n}}} \right),\\
b_2 = \frac14 P\left( {{\rho _{m,n}},{\rho _{2
m,0}}} \right) - \frac{3\rho}8  Q\left( {{\rho _{m,n}},{\rho _{2\, m,0}}} \right).\\
\quad \end{array}
\Label{}
\end{equation}
It is routine to show that
\begin{equation} \begin{split}
a &= { - 6\alpha  + \frac{{\lambda  \rho }}{{1 +  \rho}}},\\
b_1 &= \frac14\left( { - 6\alpha  + \lambda  \rho \left( {\frac{2}{{1 + 4
\rho }} - \frac{1}{{1 +  \rho }}} \right)} \right)
=\frac14\left(a-  \frac{6\lambda \rho^2 }{(1 + \rho )(1 +4 \rho )}\right),\\
b_2 &= \frac14\left( { - 6\alpha  + {\textstyle{1 \over 2}}\lambda  \rho \left( {\frac{3}{{1 + 3 \rho}} - \frac{1}{{1 +  \rho }}} \right)} \right)
=\frac14\left(a-  \frac{3\lambda  \rho^2 }{(1 + \rho )(1 +3 \rho )}\right).
\end{split}
    \Label{def.ab}
\end{equation}

\proof[Step II]
We know that the  solution of equation \eqref{eq.u} can be written as
	\[u={{u}_{c}}+{{u}_{s}},\]
where ${u}_{c}$  and ${u}_{s}$  present the fast growing and slow growing modes respectively. Assume that $E_1 \subset H_1$ represents the subspace
of all fast growing modes. By the classical center-manifold theorem (see \cite{he84}),  for all $\lambda$
sufficiently close to $\lambda_c$ , there exist a neighborhood $U_{\lambda} \subset E_1$ of $u=0$ and a $C^1$ center-manifold function $\Phi^\lambda :
U_{\lambda} \to E_1$, which depends continuously on $\lambda$, such that $u_s=\Phi^\lambda(y)$. Ma and Wang \cite{mw-b1} have developed a very strong method for asymptotic approximation of $\Phi^\lambda$.
Here we state their approximation, but refer the reader to \cite{mw-b1} for a proof.

\begin{lemma}[Approximation of the center manifold function]\label{lemma.cmf}
Assume that the control parameter, $\lambda$, of system \eqref{eq.u}  is close enough to the critical bifurcation parameter $\lambda_c$. Define
	\[{{E}_{1}}=\text{span}\{{{u}_{I}}|I\in \C\}, E_2=E_{1}^{\bot}, \mathcal{L}_1=\eL |_{E_1};\]
and let $\mathcal{P}_2:H \to E_2$ be the Leray projection. Then
	\[-\mathcal{L}_1^{-1}(\Phi (y))=\mathcal{P}_2(\G_2(\sum\limits_{I\in \C}y_I u_I))+o(|y|^2)+O(|\sigma_I| |y|^2)\]
\end{lemma}
Now by using the center manifold function approximation (Lemma~\ref{lemma.cmf}) and similar identities as in (\ref{ids}), we can obtain \vs
\begin{equation}
\begin{array}{ll}
{{y_{0,0}} =  -\tfrac{3}{8}y_{m,n}^2 - \tfrac{3}{4}y_{0,2n}^2}+o(|y|^2),
&{{y_{2m,0}} = \kappa_2 y_{m,n}^2}+o(|y|^2),\\
{{y_{m,3n}} = 4 \kappa_2 {y_{m,n}}{y_{0,2n}}}+o(|y|^2),
&{{y_{0,4n}} =2 \kappa_1 y_{0,2n}^2}+o(|y|^2),\\
{{y_{2m,2n}} = \kappa_1 y_{mn}^2}+o(|y|^2),
\end{array}
\Label{}
\end{equation}
for some constant $\kappa_2$ and $\kappa_1$ which can be calculated easily; they are given by
\begin{equation}
    \begin{split}
\kappa_1  &=  - \frac{1}{{8\sigma \left( {2m,2n} \right)}}\left( { - 6\alpha  + \frac{{4\lambda \rho }}{{\rho  + 1}}} \right),\\
\kappa_2  &=  - \frac{1}{{8\sigma \left( {2m,0} \right)}}\left( { - 6\alpha  + \frac{{3\lambda \rho }}{{\rho  + 1}}} \right).
    \end{split}
    \Label{def.kappa}
\end{equation}

Therefore,  we have
\begin{equation}
    \begin{split}
{y_1}{B_{11}} + {y_2}{B_{21}} = &y_1^3[ - \tfrac{{3a}}{8} + 2b_2\, \, \kappa_2  + b_1\, \, \kappa_1 ] + {y_1}y_2^2[8b_2\, \, \kappa_2  - \tfrac{{3a}}{4}]+ o(|y|^3),\\
{y_1}{B_{12}} + {y_2}{B_{22}} = &y_2^3[ - \tfrac{{3a}}{4} + 4b_1\, \, \kappa_1 ] + {y_2}y_1^2[4b_2\, \, \kappa_2  - \tfrac{{3a}}{8}] + o(|y|^3).
\end{split}
    \Label{}
\end{equation}

\proof[Step III]
Now we  add to the above calculations the effect of cubic interactions. For the cubic interactions we have
\begin{equation}
    \begin{split}
     \langle {{\G_3}({u_c} + {u_{c'}}),{e_k}} \rangle  &= \langle {{\G_3}({u_c}),{e_k}} \rangle  + o(|y{|^3})\\
 &=  - \alpha \langle {{{({y_1}{e_1} + {y_2}{e_2})}^3},{e_k}} \rangle + o(|y|^3).
\end{split}
\end{equation}
We  easily see that
\begin{equation}
    \begin{split}
\langle {{{({y_1}{e_1} + {y_2}{e_2})}^3},{e_{I_1}}} \rangle  &=   y_1^3\langle {e_{{I_1}}^3,{e_{{I_1}}}} \rangle  + 3 {y_1}y_2^2\langle {{e_{{I_1}}}e_{{I_2}}^2,{e_{{I_1}}}} \rangle \\
 &=    \ell_1\ell_2(\tfrac{9}{64} y_1^3 + \tfrac{3}{8} {y_1}y_2^2),\\
\langle {{{({y_1}{e_1} + {y_2}{e_2})}^3},{e_{I_2}}} \rangle &= y_2^3\langle {e_{{I_2}}^3,{e_{{I_2}}}} \rangle  +3 {y_2}y_1^2\langle {{e_{{I_2}}}e_{{I_1}}^2,{e_{{I_2}}}} \rangle \\
 &= \tfrac{3}{8} \ell_1\ell_2 ( y_2^3 +  y_2 y_1^2).
    \end{split}
    \Label{}
\end{equation}

Finally, we arrive at the following
\begin{equation}
    \begin{split}
\dot{y}_1= &4\af y_{1}^2+y_1^3[ - \tfrac{{3}}{8}a + 2b_2\,\kappa_2  + b_1\,\kappa_1 - \tfrac9{16} \alpha] + {y_1}y_2^2[8b_2\,\kappa_2  - \tfrac{{3}}{4}a- \tfrac32 \alpha]+ o(|y|^3),\\
\dot{y}_2 = &\af y_{1}y_2+y_2^3[ 4b_1\, \kappa_1 - \tfrac34 a - \tfrac34 \alpha] + {y_2}y_1^2[4b_2\, \kappa_2  - \tfrac{3}{8}a- \tfrac34 \alpha] + o(|y|^3);
\end{split}
    \Label{}
\end{equation}
where $\af$ is given in \eqref{def.A}. Now by defining
\begin{equation}
    \begin{split}
     {{\mathfrak{b}}_{1}} &=4b_1\, \kappa_1 - \tfrac34 a - \tfrac34 \alpha,\\
  {{\mathfrak{b}}_{2}} &=4b_2\, \kappa_2  - \tfrac{3}{8}a- \tfrac34 \alpha,
    \end{split}
    \Label{dif.bif}
\end{equation}
we will have \eqref{eq.rd}.
\end{proof}
\section{proof of the main theorems}
The reduction system of \ref{eq.rd} constitutes the major step for the proof of our main theorems. Here we provide a sketch of the proof; for more details we refer the reader to \cite{mw-b1,kwy09}.

\begin{proof}[Proof  of \ref{th.1}]
Given the assumptions of Theorem \ref{th.1}, we see that the system undergoes an exchange of stability at $\lambda=\lambda_c=\tfrac94\mu$ and the solutions bifurcate if $\lambda$ crosses $\lambda_c$. Now, based on Lemma \ref{lemma.rd}, system \ref{eq.u} can be reduced to \ref{eq.rd} near $\lambda_c$. On the other hand, \ref{cond.th1} implies
\[\af=0,\quad \bif_1= -{\frac {21}{80}}\,\mu,\quad \bif_2=-{\frac {57}{128}}\,\mu.\]
Properties of this type of equations are discussed in \cite{mw-b1,kwy09}. It is straightforward to show that the trivial solution is an asymptotically stable equilibrium point when $\lambda<\lambda_c$. If $\lambda>\lambda_c$, it is shown in \cite{mw-b1} that there is  a basin of attraction, and an attractor which consists of finite number of steady-state solutions and their connecting heteroclinic orbits. Nontrivial steady-state solutions are solutions of the truncated stationary equation \ref{eq.rd}. They are given by

\begin{equation}
\begin{split}
\pm Y_s& =\pm (0,\sqrt{\frac{-\sigma }{{{\mathfrak{b}}_{1}}}}), \\
 \pm Y_r& =\pm (\sqrt{\frac{-8\sigma }{{{\mathfrak{b}}_{1}}+2{{\mathfrak{b}}_{2}}}},0), \\
\pm Y_h^{\pm} & =\pm (\pm 2 ,
 1) \sqrt{\frac{-\sigma}{\mathfrak{b}_{1}+4\mathfrak{b}_2}}.
    \end{split}
    \Label{def.ep}
\end{equation}
with $\sigma=\sigma_{m,n}$.

A straightforward calculation shows that when $\lambda>\lambda_c$, the trace of the Jacobian matrix, $J,$ for any of the equilibrium points is negative, but $\sg(\det(J))=\sg(\bif_1-2\bif_2)$ for strips and rectangular patterns, and $\sg(\det(J))=-\sg(\bif_1-2\bif_2)$ for hexagonal patterns. Obviously we have $\bif_1-2\bif_2>0$, therefore  hexagonal patterns are stable nodes and all other equilibrium points are saddle points. This completes the proof of the theorem.
\end{proof}

\begin{proof}[Proof  of \ref{th.2}]
Assuming the conditions of Theorem \eqref{th.2} results in having an exchange of stability at $(\ell_c,\lambda_c)$, Lemma\eqref{lemma.pes}. Therefore the system undergoes a phase transition at $(\ell_c,\lambda_c)$. To determine the type of transition, we reduce the system to its center manifold near $(\ell_c,\lambda_c)$ to extract equations \ref{eq.rd}. From the proof of Theorem \ref{th.1}, we know that $\bif_1(\lambda,\ell),\bif_2(\lambda,\ell)<0$ near critical parameter values. Therefore, again we will have an $S^1$ bifurcation from the asymptotically stable trivial solution as $(\ell,\lambda)$  crosses $(\ell_c,\lambda_c)$. Therefore, the transition is a Type I transition.   Now  steady-state solutions of \eqref{eq.rd} are  solutions of the following stationary truncated equation
\begin{equation}
    \begin{split}
&\sigma_1 y_1+ 4\mathfrak{a}y_{1}^{2}+\tfrac{1}{4}({{\mathfrak{b}}_{1}}+2{{\mathfrak{b}}_{2}})y_{1}^{3}
+2{{\mathfrak{b}}_{2}}{{y}_{1}}y_{2}^{2}=0, \\
&\sigma_2 y_2+\mathfrak{a}{{y}_{1}}{{y}_{2}}+{{\mathfrak{b}}_{1}}y_{2}^{3}+{{\mathfrak{b}}_{2}}{{y}_{2}}y_{1}^{2}=0.
    \end{split}
    \Label{eq.rdt}
\end{equation}
It is easy to see that when one pair of solutions is given by
\[\pm Y_s  =\pm (0,\sqrt{\frac{-\sigma }{{{\mathfrak{b}}_{1}}}}).\]
Since $\af\ne0$, we have no rectangular pattern. It is straightforward to see that we can have $y_1=\pm2y_2$; therefore, there exist four hexagonal patterns. Also we have two solution with $y_2=\frac{4a}{b_1-2b_2}.$ These solutions correspond to mixed patterns as shown in Fig.\ref{fig.II}.
\end{proof}

\end{document}